\documentclass[preprint,sort&compress,12pt]{elsarticle}
\usepackage[hidelinks, colorlinks=true]{hyperref}
\usepackage{amscd,amssymb,amsmath,amsthm}
\usepackage[all]{xy}
\usepackage{array}
\usepackage{etoolbox}
\usepackage{enumitem}
\usepackage{mathtools}
\usepackage{siunitx}
\usepackage{multirow}
\usepackage{tikz-cd}
\usepackage{bm}
\usepackage{tikz}
\usepackage{tabularx}
\usepackage{makecell}
\usetikzlibrary{matrix,fit}
\usetikzlibrary{positioning}
\usetikzlibrary{shapes.geometric}
\usepackage{caption}
\usepackage[tickmarkheight=0.1cm]{todonotes}

\newtheorem{theorem}{Theorem}[section]

\newtheorem{lemma}[theorem]{Lemma}
\newtheorem{corollary}[theorem]{Corollary}
\newtheorem{proposition}[theorem]{Proposition}
\newtheorem{remark}[theorem]{Remark}

\numberwithin{equation}{section}
\theoremstyle{definition}
\newtheorem{definition}[theorem]{Definition}

\newtheorem*{Outline}{Outline of the paper}

\newcommand{\Z}{\mathbb{Z}}

\newlength\mylen
\usepackage{array} 
\newcolumntype{C}{>{\hfil$}p{\mylen}<{$\hfil}}

\newcommand\dela[1]{}

\usepackage[utf8]{inputenc}

\makeatletter
\def\ps@pprintTitle{%
	\let\@oddhead\@empty
	\let\@evenhead\@empty
	\def\@oddfoot{}%
	\let\@evenfoot\@oddfoot}
\makeatother

\begin{document}
	
	\begin{frontmatter}

		\title{On the Schur multiplier of $p$-groups with abelianization $s$-elementary abelian}
		\author[NISER,HBNI]{Sumana Hatui}
            \ead{sumanahatui@niser.ac.in}
		\author[NISER,HBNI]{Tony Nixon Mavely}
		\ead{tonynixonmavely@gmail.com}
		\author[NISER,HBNI]{Sahanawaj Sabnam}
		\address[NISER]{School of Mathematical Sciences,  National Institute of Science Education and Research Bhubaneswar, 752050, Odisha, India.}
            \address[HBNI]{Homi Bhabha National Institute, Training School Complex, Anushakti Nagar, Mumbai, 400094, India}
		\ead{sahanawaj.sabnam@niser.ac.in}
        \ead{sahanawajsabnam42@gmail.com}
		
		\begin{abstract}
			Let $p$ be an odd prime. We describe a method to compute the Schur multiplier of finite \(p\)-groups \(G\) of nilpotency class \(2\) such that
${G/[G,G] \cong (\mathbb{Z}_{p^{s}})^{k}, s,k \in \mathbb{N}}$,
generalizing a method of Blackburn and Evens, who treated the case \(s=1\).
As an application, we investigate which abelian \(p\)-groups can occur as the Schur multiplier of a non-abelian \(p\)-group, a problem posed in Kourakava’s notebook.
We further introduce the notions of \(s\)-special \(p\)-groups of rank $k$ generalizing the notion of special $p$-groups of rank $k$.
We study the structural properties, compute the Schur multipliers of \(s\)-special \(p\)-groups of rank $1$.

		\end{abstract}
		
		
	\end{frontmatter}
\section{Introduction}
In this article,  we develop an approach to determine the Schur multiplier for an extensive family of finite 
$p$-groups where $p$ is an odd prime. Let $G$ be a finite $p$-group and $G'$ denotes the commutator subgroup of $G$. The Schur multiplier $M(G)$ of a group $G$ is the second cohomology group $H^2(G,\mathbb{C}^\times)$. 
In \cite{BlaEve1979}, Blackburn and Evens determined the Schur multiplier of $p$-groups $G$ of nilpotency class $2$ whose abelianization $G/G'$ is
elementary abelian, and as a notable consequence, computed the Schur
multiplier of extraspecial $p$-groups. 
Motivated by this work, we generalize these results to a broader class of groups $G$ of
nilpotency class $2$ satisfying $G/G' \cong (\mathbb{Z}_{p^s})^{k}$ with $s,k \in \mathbb{N}$. We call a finite abelian \(p\)-group an \emph{\(s\)-elementary abelian \(p\)-group}
if it is isomorphic to \((\mathbb{Z}_{p^{s}})^{k}\); such groups are
precisely the homocyclic \(p\)-groups of exponent \(p^{s}\).
This notion generalizes that of elementary abelian \(p\)-groups.

The following notation will be used throughout the paper. For two abelian groups $A$ and $B$, $A\otimes B$ denotes the tensor product of $A$ and $B$ as $\mathbb Z$-modules and $A\wedge A$ is the group defined by $\frac{A\otimes A}{\langle a\otimes a\mid a \in A \rangle}$. For ${a_1, a_2\in A}$, let $a_1\wedge a_2$ denote the image of $a_1\otimes a_2$ in $A \wedge A$.
For a $p$-group $G$, we define ${V:=G/G'}$ and $W:=G'$.  
Denote ${(v_1,v_2)=[g_1,g_2]}$ for ${v_1=g_1G'}$ and ${v_2=g_2G'}$. Let $X_1$ be the subgroup of ${V \otimes W}$ generated by the elements $$(v_1,v_2,v_3):=v_1 \otimes (v_2,v_3)+v_2 \otimes (v_3,v_1) + v_3 \otimes (v_1,v_2),v_1,v_2,v_3 \in V,$$ and $X_2$ be the subgroup of $V \otimes W$ generated by the elements $v \otimes f(v)$ for $v\in V$, where $f:V \to W$ is defined by  $f(v)=g^{p^s}$ for $v=gG'$. Let $X$ be the subgroup of $V \otimes W$ generated by $X_1$ and $X_2$.
\begin{theorem} \label{maintheorem Schur computation}
    Let $G$ be a $p$-group with nilpotency class $2$, and $G/G'$ be an $s$-elementary abelian $p$-group. Define a homomorphism
\[
\rho : V \wedge V \longrightarrow W
\text{ by }
\rho(v_{1} \wedge v_{2}) = (v_{1}, v_{2})
\quad \text{for all } v_{1}, v_{2} \in V.
\]
The following holds:
\begin{enumerate}
    \item There exists an abelian group $M^*$ with a subgroup $N$ such that 
    \[
    N \cong \frac{V \otimes W}{X}, \qquad
\frac{M^{*}}{N} \cong V \wedge V.
\]
    \item Let $M$ be the subgroup of $M^{*}$ containing $N$ such that $M/N$
corresponds to $\ker \rho$. Then
$M(G) \cong M.$

\end{enumerate}
\end{theorem}

The proof of this result is presented in Section~\ref{section: Schur multiplier of group with s elementary abelian}.
As a further consequence of Theorem \ref{maintheorem Schur computation}, we address a problem posed in the Kourovka notebook \cite[Problem 15.30]{Kour2018}, which is given below.

\medskip
\noindent\textbf{Problem.}
Is it true that every finite abelian \(p\)-group is isomorphic to the Schur
multiplier of some non-abelian finite \(p\)-group?

\medskip
Recently, Rai \cite[Theorem 6]{Rai2023}, proved that every elementary abelian \(p\)-group arises as the Schur multiplier of some
non-abelian finite \(p\)-group. Generalizing this result, we prove the following
theorem in Section \ref{Kourovka}.
\begin{theorem}\label{Application_Kourovka}
    Let $s\in \mathbb{N}$. The following abelian groups occur as the Schur multiplier of some non-abelian $p$-groups.
    \begin{enumerate}

       \item [(i)]$ (\Z_{p^s})^n \times \Z_{p^{m_1}}$, where $m_1,n \in \mathbb{N} \cup \{0\}$ and $ m_1 < s $.
    \item [(ii)] $(\Z_{p^s})^n \times \Z_{p^{m_1}} \times \Z_{p^{m_2}} $, where $m_1,m_2 \in \mathbb{N} \cup \{0\}$, $n \in \mathbb{N} \setminus \{2\}$ and $ m_1,m_2 < s $.
    \item [(iii)]  $(\Z_{p^s})^n \times \Z_{p^{m_1}} \times \Z_{p^{m_2}} \times \ldots \times \Z_{p^{m_r}} $, where $3 \leq r $,  $m_i < s$ for $1 \le i \le r$ and $n \ge \frac{1}{2}(a-1)(a-2)$ for $a=\lceil\frac{3r}{2}+2\rceil$.
    \end{enumerate}
    In particular, the group $\Z_{p^{n_1}} \times \Z_{p^{n_2}} \times \Z_{p^{n_3}}$ is also obtained as the Schur multiplier of some non-abelian $p$-group, for any $n_1,n_2,n_3 \in  \mathbb{N} \cup \{0\}$ .
\end{theorem}
A proof of this result is given in Section \ref{Kourovka}.
\begin{remark}
For distinct primes $p_i$, let $H(p_i)$ denote a group listed in Theorem \ref{Application_Kourovka}. By \cite[Theorem 2.2.10]{Karp1987}, $\prod_{i=1}^{n}H({p_i})$ is also the Schur multiplier of some non-abelian group.

    \end{remark}

A finite \(p\)-group \(G\) is called an
\emph{\(s\)-special \(p\)-group of rank k} if
$G' \cong Z(G)$ is an $s$-elementary abelian $p$-group minimally generated by $k$ elements
and \(G/G'\) is an \(s\)-elementary abelian \(p\)-group. We say that a group is an \emph{$s$-extraspecial} $p$-group if it is an $s$-special $p$-group of rank $1$. These groups genralize the notion of special $p$-groups of rank $k$ and extraspecial $p$-groups. The structure of extraspecial \(p\)-groups was subsequently described in
\cite[Theorem 3.3.4]{Karp1987}, where it was shown that every such group is a central product
of extraspecial \(p\)-groups of order \(p^{3}\). Recall that a group $G$ is said to be a central product of subgroups $G_i$, $i \in \mathcal{I}$, amalgamating $D$ if $G_i$ generate $G$ and for all $i \neq j$ in $\mathcal{I}$, $G_i \cap G_j =D$ and $[G_i,G_j]=1$. We obtain the structure of $s$-extraspecial $p$-groups and compute their Schur multiplier in the following result. This generalizes the result for extraspecial $p$-groups.
\begin{theorem}\label{Centralproductstructure}
Let \(G\) be an \(s\)-extraspecial \(p\)-group, where \(s \in \mathbb{N}\).
Then there exists an integer \(r \geq 1\) such that \(G\) is the central product
of \(r\) many \(s\)-extraspecial \(p\)-subgroups of order \(p^{3s}\), amalgamating
the center \(Z(G)\).
In particular,
$
|G| = p^{(2r+1)s}$.
Moreover, if \(r > 1\), then $M(G) \cong (\mathbb{Z}_{p^{s}})^{\,2r^{2}-r-1}$.

\end{theorem}
A proof of this result is given in Section~\ref{s-extraspecial structure}.
Furthermore, we show that for \(r>1\), these groups are unicentral (c.f. Definition \ref{definition: capable, epicenter, unicentral}).


\begin{Outline} In Section 2, we recall some definitions and results that will be used in the paper. In Section 3, we present a method for computing the Schur multiplier of $p$-groups $G$ of nilpotency class $2$ with $G/G'$ being $s$-elementary abelian. In Section 4, as an application of this method we show that many abelian groups can be achieved as the Schur multiplier of non-abelian $p$-groups. In Section 5, we obtain the structure and the Schur multiplier of $s$-extraspecial $p$-groups. 
\end{Outline}
\noindent\textbf{Notation.}
The center of a group \(G\) is denoted by \(Z(G)\). The commutator of two elements $a,b \in G$  is the element $[a,b]=a^{-1}b^{-1}ab$.
The commutator subgroup \(G'\) is also written as \(\gamma_{2}(G)\), and more generally
\(\gamma_{i}(G)\) denotes the \(i\)th term of the lower central series of \(G\).
The
minimal number of generators of \(G\) is denoted by \(d(G)\). The order of an element $g$ in $G$ is denoted by $o(g)$. The group of homomorphisms from $A$ to $B$ is denoted by $\text{Hom}(A,B)$.
We say $\mu:B \to A$ is a section of the map $s: A \to B$ if  $s \circ \mu=1_B$. In a presentation of a group $G$, 
we omit all the relations of the form $[x,y]=1$ for the generators $x,y$ of $G$. The cyclic group of order $n$ will be denoted by $\mathbb{Z}_n$.
\section{Preliminaries}
In this section, we state and prove the results that will be used in the subsequent sections.
We recall the following definition from \cite{Doug1951}, which provides a notion of a
basis for a finite abelian group and will be used throughout the paper.
\begin{definition} \label{definition of basis}
    Let \(G\) be a finite abelian group.
A subset \(\{g_{1},\ldots,g_{d}\} \subseteq G\) is said to be \emph{complete} if
every element \(g \in G\) can be written in the form
$g = g_{1}^{a_{1}} \ldots g_{d}^{a_{d}},$
where \(a_{i} \in \mathbb{Z}\) for all \(1 \leq i \leq d\).
The set is said to be \emph{independent} if
$g_{1}^{a_{1}} \ldots g_{d}^{a_{d}} = 1,
~ 0 \leq a_{i} < o(g_{i}),$
implies \(a_{i} = 0\) for all \(1 \leq i \leq d\).
A \emph{basis} of \(G\) is a subset \(\{g_{1},\ldots,g_{d}\}\) that is both
complete and independent.
 
\end{definition}

\begin{remark}
    The above definition of independence of a set with respect to a finite abelian group can differ from the notion of linear independence in an $R$-module. But if $G$ is an $s$-elementary abelian $p$-group viewed as a $\Z_{p^s}$-module, then the definitions coincide. Throughout the paper, when we say that a set is independent (or a basis), we are using the above definition.
\end{remark}

\begin{lemma} (\cite[Main result]{Doug1951})\label{existence of basis for finite abelian group}
    If $G$ is a finite abelian group, then $G$ has a basis.
\end{lemma}
Note that, a minimal generating set is a basis for a finite abelian group.

    Let $G$  be a group and $A$ be an abelian group. A central extension of $A$ by $G$ is a short exact sequence 
\[
\begin{tikzcd}
  E:& 1 \arrow[r] & A \arrow[r,"\alpha"] & B \arrow[r, "\beta"] & G \ar[r] & 1
\end{tikzcd},
\]
such that $\alpha(A) \subseteq Z(B)$. Two extensions $E$ and $E'$ are congruent if there exists a homomorphism $\gamma:B \to C$ which renders the following diagram commutative
\[
\begin{tikzcd}
  E: & 1 \arrow[r] & A \arrow[d, "1_A"] \arrow[r, "\alpha"] & B \arrow[d, "\gamma"] \arrow[r, "\beta"] & G \arrow[d, "1_G"] \arrow[r] & 1\\
  E': &1 \arrow[r] & A \arrow[r, "\alpha'"] & C \arrow[r, "\beta'"] & G \ar[r] & 1.
\end{tikzcd}
\]
If $B$ is an abelian group, then the extension $E$ is called an abelian extension.
We define Ext$(G,A)$ to be the group of congruence classes of abelian extensions of $A$ by $G$.

\begin{remark} (cf \cite[pg 27]{Karp1987}) \label{Hom and Ext isomorphism for finite abelian groups}
    For any two finite abelian groups $A$ and $B$, $\mathrm{Hom}(A,B) \cong \mathrm{Ext}(A,B)$. 
\end{remark}


\begin{theorem}(\cite[Theorem 3.2.1]{Karp1987})\label{Blackburn general  exact sequence}
    Suppose that $N$ is a normal subgroup of a finite group $G$. If $F$ is a free group of finite rank, $R$ is a normal subgroup of $F$ for which $G=F/R$ and $S$ is a normal subgroup of $F$ for which $SR/R$ corresponds to $N$, then, there is an exact sequence
    \begin{equation}
\begin{tikzcd}
 1 \arrow[r] & \frac{R \cap [F,S]}{[F,R] \cap [F,S]} \arrow[r] & M(G) \arrow[r] & M(\frac{G}{N}) \arrow[r] & \frac{N \cap G'}{[N,G]} \arrow[r] & 1.
\end{tikzcd}
\label{eq:Blackburn general exact sequence}
\end{equation}
\end{theorem}

\begin{lemma}(\cite[Corollary 3.2.3]{Karp1987}) \label{Karpilovsky exact sequence}
    Let $G$ be a finite nilpotent group of class $c \geq 2$ and let $G=F/R$ be a presentation of $G$ as a factor group of a free group of finite rank. Then, there is an exact sequence
    \begin{equation}
\begin{tikzcd}
 1 \arrow[r] & \frac{\gamma_{c+1}(F)}{[F,R] \cap \gamma_{c+1}(F)} \arrow[r] & M(G) \arrow[r] & M(\frac{G}{\gamma_c(G)}) \arrow[r] & \gamma_c(G) \arrow[r] & 1.
\end{tikzcd}
\label{Blackburn exact sequence}
\end{equation}
\end{lemma}

\begin{lemma}( \cite[Corollary 3.2.4]{Karp1987})\label{Karpilovsky lambda map}
    Let $G$ be a finite nilpotent group of class $c \geq 2$ and let $G=F/R$ be a presentation of $G$ as a factor group of a free group of finite rank. Then, 
    \begin{enumerate}
        \item [(i)] Upon identifying $G/G'$ with $F/F'R$ and $\gamma_c(G)$ with $\gamma_c(F)R/R$, the map
        \begin{align*}
            G/G' \otimes \gamma_c(G) &\xrightarrow{\lambda} \gamma_{c+1}(F)/([F,R] \cap \gamma_{c+1}(F)) \\
            fR \otimes xR &\mapsto [f,x]([F,R] \cap \gamma_{c+1}(F))
        \end{align*}
        is a surjective homomorphism.
        \item [(ii)] Let $\phi:G/G' \otimes \gamma_c(G) \to M(G)$ be the composition of $\lambda$ with the homomorphism $\gamma_{c+1}(F)/([F,R] \cap \gamma_{c+1}(F)) \to M(G)$ in \eqref{Blackburn exact sequence}. Then the sequence
            \begin{equation}\label{Blackburn exact sequence with X}
    1 \to Ker \lambda \to G/G' \otimes \gamma_c(G) \xrightarrow{\phi}\ M(G) \to M(\frac{G}{\gamma_c(G)}) \to \gamma_c(G) \to 1.
\end{equation}
is exact. In particular, $M(G)$ is an extension of $(G/G' \otimes \gamma_c(G))/Ker \lambda$ by $Ker \ (M(G/\gamma_c(G)) \to \gamma_c(G))$, i.e., the following sequence is exact.
    \end{enumerate}
    \[
\begin{tikzcd}
 1 \arrow[r] & \frac{\gamma_{c+1}(F)}{[F,R] \cap \gamma_{c+1}(F)} \arrow[r] & M(G) \arrow[r] & M(\frac{G}{\gamma_c(G)}) \arrow[r] & \gamma_c(G) \arrow[r] & 1.
\end{tikzcd}
\]
\end{lemma}

Let \(Z\) be a central subgroup of a finite group \(G\). The following exact sequence (cf  \cite[Theorem 2.5.6] {Karp1987}) is known as the \emph{Ganea exact sequence}:
\[
\begin{tikzcd}
G/G' \otimes Z \arrow[r,"\phi_Z"] &
M(G) \arrow[r] &
M(G/Z) \arrow[r] &
G' \cap Z \arrow[r] &
1.
\end{tikzcd}
\]
\begin{definition} \label{definition: capable, epicenter, unicentral}
    A group $G$ is said to be capable if $G \cong E/Z(E)$ for some group $E$. The epicenter $Z^*(G)$ of a group $G$ is the smallest central subgroup of $G$ such that $G/Z^*(G)$ is capable. A group $G$ is unicentral if $Z^*(G)=Z(G)$.
\end{definition}
The following Lemma follows from \cite[Lemma 4.2]{Beyl1979} and the above exact sequence.
\begin{lemma}\label{capability 1}
Let \(Z\) be a central subgroup of a finite group \(G\).
Then \(Z \subseteq Z^{*}(G)\) if and only if
$(G/G') \otimes Z = \ker(\phi_Z).$
\end{lemma}
\begin{lemma}\label{capability 2}
Let \(G\) be a finite group of nilpotency class \(2\) and \(Z\) be a subgroup
of \(G\) such that \(Z \subseteq Z(G) \cap G'\).
Then \(Z \subseteq Z^{*}(G)\) if and only if
\[
(G/G') \otimes Z \subseteq \ker(\lambda).
\]
\end{lemma}
\begin{proof}
    The result follows from Lemma~\ref{Karpilovsky lambda map} and Lemma~\ref{capability 1} as $\phi_Z=\lambda|_{(G/G') \otimes Z}$. 
\end{proof}
    
\section{Schur multiplier of $p$-groups $G$ of nilpotency class $2$ with $G/G'$ being $s$-elementary abelian} 
\label{section: Schur multiplier of group with s elementary abelian}
In this section, we present a method for computing the Schur multiplier of
\(p\)-group $G$ of nilpotency class \(2\) with \(G/G'\) being
\(s\)-elementary abelian.
For such groups, \(G'\) has exponent \(p^{t}\) for some
\(t \leq s\).
The methods given in this section generalize the arguments in
\cite[Section~3.3]{Karp1987}.
\begin{lemma} \label{homomorphism to extension}
    Let $A$ be an $s$-elementary abelian $p$-group and $B$ be an abelian $p$-group with exponent $p^t$ such that $t \leq s$, where $t,s \in \mathbb{N}$. If $\sigma:A \to B$ is a homomorphism, then there exists an abelian extension
    \[
\begin{tikzcd}
  1 \arrow[r] & B \arrow[r] & C \arrow[r, "\phi"] & A \ar[r] & 1
\end{tikzcd}
\]
unique up to congruency such that $\sigma \circ \phi(c)=c^{p^s}$ for all $c \in C$.
\end{lemma}
\begin{proof}
        Let $E$ be an arbitrary abelian extension
\[
\begin{tikzcd}
  E:& 1 \arrow[r] & B \arrow[r,"i"] & C \arrow[r, "\phi"] & A \ar[r] & 1
\end{tikzcd}
\]
of $B$ by $A$. We define a map $\sigma_{E}:A \to B$ by setting $\sigma_E(a)=c^{p^s}$, where $c \in C$ such that $a=\phi(c)$. 
It is easy to see that $\sigma_{E}$ is a well defined homomorphism and depends only on the congruency class of $E$. Therefore, $\{E\} \mapsto \sigma_E$ defines a map from $\mathrm{Ext}(A,B)$ to $\mathrm{Hom}(A,B)$. We have $\mathrm{Ext}(A,B) \cong \mathrm{Hom}(A,B)$ by Remark \ref{Hom and Ext isomorphism for finite abelian groups}. Hence, it is enough to show that the map is an injective homomorphism. We first show that it is a homomorphism. Let
\[
\begin{tikzcd}
  E_i:& 1 \arrow[r] & B \arrow[r] & C_i \arrow[r, "\phi_i"] & A \ar[r] & 1 & (i=1,2)
\end{tikzcd}
\]
be two extensions of $B$ by $A$ and let $\mu_i$ denote a section of $\phi_i$. 
Each element $c_i \in C_i$ can be uniquely written in the form $b\mu_i(a)$ for some $a \in A$ and $b \in B$ and so $c_i=b\mu_i(a)$. Therefore, we can write ${\sigma_{E_i}(a)=\sigma_{E_i}(\phi_i(b\mu_i(a)))=\mu_i(a)^{p^s}}$. Note that $\alpha_i:A \times A \to B$ defined by $\alpha_i(x,y)=\mu_i(x)\mu_i(y)\mu_i(xy)^{-1}$ is a cocycle in $Z^2(A,B)$ corresponding to $E_i$. Let $S$ and $T$ be subgroups of $C_1 \times C_2$ and $B \times B$ respectively defined as follows:
\[
S=\{(c_1,c_2) \in C_1 \times C_2 \mid \phi_1(c_1)=\phi_2(c_2)\},
\]
\[
T=\{(b,b^{-1})\mid b \in B\}.
\]
The map $\frac{(B \times B)}{T} \to B$ defined as $(b_1,b_2)T \mapsto b_1b_2$ is an isomorphism. We consider the following extension:
\[
\begin{tikzcd}
  E_3:& 1 \arrow[r] & \frac{B \times B}{T} \arrow[r] & S/T \arrow[r, "\phi_3"] & A \ar[r] & 1 
\end{tikzcd}
\]
with $\phi_3((c_1,c_2)T)=\phi_i(c_i)$ where $c_i \in C_i$. This sequence is exact and the map $\mu_3:A \to S/T$ defined by $\mu_3(a)=(\mu_1(a),\mu_2(a))T$ is a section of $\phi_3$. Let $\alpha_3$ be the cocycle corresponding to $E_3$ given by $\alpha_3(x,y)=\mu_3(x)\mu_3(y)\mu_3(xy)^{-1}$. 
    Then we have $\alpha_3=\alpha_1\alpha_2$ and $\mu_3(a)^{p^s}=(\mu_1(a),\mu_2(a))^{p^s}T=(\mu_1(a)^{p^s},\mu_2(a)^{p^s})T$ which maps to $ \mu_1(a)^{p^s}\mu_2(a)^{p^s}$ in $B$. This implies that $\sigma_{E_3}(a)=\mu_3(a)^{p^s}=\mu_1(a)^{p^s}\mu_2(a)^{p^s}=\sigma_{E_1}(a)\sigma_{E_2}(a)$. Therefore, the map $\{E\} \mapsto \sigma_E$ is a homomorphism. To show that injectivity holds, assume that $\sigma_E=1$. This implies that $c^{p^s}=1$. Thus, the exponent of $C$ is at most $p^s$. Let $\{a_1,\ldots,a_d\} \subset C$ denote preimages of a minimal generating set of  $A \cong C/B \cong (\mathbb{Z}_{p^s})^d$. If $o(a_i) < p^s$, then $a_iB$ cannot have order $p^s$. Therefore, the subgroup $\langle a_1,\ldots,a_d \rangle$ of $C$ is isomorphic to $A$. Consider the set $\{b_1,\ldots,b_k\}$ to be the images in $C$ of a minimal generating set of $B$. We have the subgroup $\langle b_1,\ldots,b_k \rangle$ of $C$ is isomorphic to $ B$. Note that the sets $\{ a_1,\ldots,a_d \}$, $\{b_1,\ldots,b_k\}$ are bases for the corresponding subgroups of $C$. We will show that the intersection of these two subgroups of $C$ is trivial. Suppose not, then we have $b_1^{m_1}\ldots b_k^{m_k}=a_1^{n_1}\ldots a_d^{n_d}$ where $0 \leq m_i < p^{t_i}$ for $1 \leq i\leq k$ and $0 \leq n_j < p^s$ for $1 \leq j \leq d$. This gives us $1=\phi(a_1)^{n_1}\ldots \phi(a_d)^{n_d}$ which implies that $n_j \equiv 0 \pmod {p^s}$ for all $j$ with $1 \leq j \leq d$. Thus, $b_1^{m_1}\ldots b_k^{m_k}=1$ which gives $m_i \equiv 0 \pmod {p^{t_i}}$ for all $i$ with $1\leq i \leq k$. Thus, the sequence $E$ is a split exact sequence and the map is injective. 
 \end{proof}

 \begin{lemma} \label{required homomorphism}
     Let $G$ be a $p$-group of nilpotency class $2$ with $G/G'$ an $s$-elementary abelian $p$-group, and let $f:V \to W$ be the map defined by ${f(gG')=g^{p^s}}$.
     \begin{enumerate}
         \item[(i)] The map $f$ is a group homomorphism.
         \item[(ii)] The map $\sigma: V \wedge V \to \frac{V \otimes W}{X}$ given by $\sigma(v_1 \wedge v_2)=(v_1 \otimes f(v_2))+X$ is a homomorphism.
     \end{enumerate}
 \end{lemma}

 \begin{proof}
     \begin{enumerate}
         \item[(i)] Let $v_1=g_1G'$ and $v_2=g_2G'$ and $v_1 v_2=g_1g_2G'$. Since ${\gamma_3(G)=1}$, we have $(g_1g_2)^{p^s}=g_1^{p^s}g_2^{p^s}c^{\binom{p^s}{2}}$ using the Hall-Petrescu identity, where $c \in \gamma_2(G)$. Since $p$ is odd, $p^s$ divides $\binom{p^s}{2}$ and therefore, $c^{\binom{p^s}{2}}=1$. Thus, $(g_1g_2)^{p^s}=g_1^{p^s}g_2^{p^s}$ which gives $f(v_1v_2)=f(v_1)f(v_2)$.
         \item[(ii)] Let $\lambda: V \times V \to \frac{V \otimes W}{X}$ be defined by $\lambda(v_1,v_2)=\sigma(v_1 \wedge v_2)$. Note that $\lambda(v,v)=(v \otimes f(v))+X=X$ since $v \otimes f(v) \in X_2 \subseteq X$. Therefore, it suffices to verify that $\lambda$ is bilinear. By $(i)$, it follows that $\lambda$ is bilinear.
     \end{enumerate}
 \end{proof}
\textbf{Proof of Theorem \ref{maintheorem Schur computation}.}
\begin{enumerate}
    \item 
Note that if $V$ is an $s$-elementary abelian $p$-group, then ${V \wedge V}$
is also an $s$-elementary abelian $p$-group. Using
Lemma~\ref{required homomorphism}(ii) and
Lemma~\ref{homomorphism to extension}, there exists an abelian group
$M^{*}$ with a subgroup $N$ such that
\[
N \cong \frac{V \otimes W}{X}, \qquad
\frac{M^{*}}{N} \cong V \wedge V.
\]
Moreover, if $xN$ (for $x \in M^{*}$) corresponds to $u \in V \wedge V$,
then $x^{p^{s}}$ corresponds to $\sigma(u)$.

\item We split the proof into two steps:
    
\noindent    \textit{Step 1.} We show that $|M(G)| \leq |M|$.
    Let $G \cong F/R$ be a presentation of $G$, where $F$ is a free group of
finite rank and $R$ is a normal subgroup of $F$. Using Lemma \ref{Karpilovsky exact sequence}, the following sequence is exact: 
\[
\begin{tikzcd}
 1 \arrow[r] & \frac{\gamma_3(F)}{[F,R] \cap \gamma_3(F)} \arrow[r] & M(G) \arrow[r] & M(G/G') \arrow[r] & G' \arrow[r] & 1.
\end{tikzcd}
\]
Hence $|M(G)|=\frac{|M(G/G')|}{|G'|}|\frac{\gamma_3(F)}{[F,R] \cap \gamma_3(F)}|$.
By Corollary~\ref{Karpilovsky lambda map}, there is a
well-defined homomorphism
$\lambda : V \otimes W \to
\gamma_{3}(F)/\bigl([F,R] \cap \gamma_{3}(F)\bigr)$
given by 
\[
\lambda(v \otimes w) =
[f,f']\bigl([F,R] \cap \gamma_{3}(F)\bigr),
\]
where $f \in F$ and $f' \in F'$ are preimages of $v$ and $w$, respectively,
under the natural isomorphism $F/R \cong G$.
Given this homomorphism, it follows that $\lambda(v_1 \otimes (v_2,v_3))=[f_1,[f_2,f_3]]([F,R] \cap \gamma_3(F))$.
Using the Hall-Witt identity \cite[Lemma 3.1.7]{Karp1987}, we have $$[f_1,[f_2,f_3]][f_2,[f_3,f_1]][f_3,[f_1,f_2]] \in \gamma_4(F) \subseteq [F,R] \cap \gamma_3(F).$$
Therefore, $X_1 \subseteq Ker \ \lambda$. If $v=gG'$ and $f \in F$ correspond to $g$, then $\lambda(v \otimes f(v))=[f,f']([F,R] \cap \gamma_3(F))$, where $f' \in F'$ and $f^{-p^s}f' \in R$. Therefore $[f,f']=[f,f^{-p^s}f'] \in [F,R] \cap \gamma_3(F)$,
proving that ${X_2 \subseteq Ker \ \lambda}$. Since $\lambda$ is surjective, it follows that
\[
\bigl|\gamma_{3}(F)/([F,R] \cap \gamma_{3}(F))\bigr|
\le \left|\frac{V \otimes W}{X}\right|.
\]
Using the fact that $M(G/G') \cong V \wedge V$, we obtain
\[
|M(G)|
\le \frac{|V \wedge V|}{|W|}
\left|\frac{V \otimes W}{X}\right|
= \left|\frac{M}{N}\right|\,|N|
= |M|.
\]

\textit{Step 2.} We proceed by constructing a group $D$ with $M^{*} \le Z(D)$ such that
\[
D/M^{*} \cong G \quad \text{and} \quad M^{*} \cap D' = M.
\]
After establishing this construction, Theorem~\ref{Blackburn general exact sequence}
provides a surjective homomorphism from $M(G)$ onto $M$. By Step~1,
we have $|M(G)| \le |M|$, and hence this homomorphism is an isomorphism.

Let $\{a_1G',\ldots,a_dG'\}$ and $\{b_1,\ldots,b_k\}$ be minimal generating sets of $V$ and $W$ respectively. Then $G$ has defining relations of the form
\begin{align*}
    [a_i,a_j]&=\prod \limits_{n=1}^k b_n^{\gamma_{ij}^n} & (i,j=1,2,\ldots,d), \\
    a^{p^s}_i&=\prod \limits_{n=1}^k b_n^{\alpha_{in}} & (i=1,2,\ldots,d),\\
    [b_i,a_l]&=[b_i,b_j]=b_i^{p^{t_i}}=1 & \text{ where } t_i \leq s \  \forall \ i=1,2,\ldots,k; \\
    & & l= 1,2, \ldots, d.
\end{align*}
For $1 \le i \le k$ and $1 \le j \le d$, choose elements $d_{ij} \in N$
corresponding to $a_jG' \otimes b_i + X$. Further, for $1 \le i,j \le d$,
let $e_{ij} \in M^{*}$ be elements such that $e_{ij}N$ corresponds to
$a_jG' \wedge a_iG'$, with $e_{ij}e_{ji} = e_{ii} = 1$.
By construction, $e_{ij}^{p^{s}} \in N$, and it corresponds to
\[
\sigma(a_jG' \wedge a_iG')
= a_jG' \otimes a_i^{p^{s}} + X.
\]
It follows from the definitions of $X, \ d_{ij}$, and $e_{ij}$ that
$$a_jG' \otimes \prod \limits _n b_n^{\gamma_{il}^n}-a_iG' \otimes \prod \limits _nb_n^{\gamma_{jl}^n}-a_lG' \otimes \prod \limits _n b_n^{\gamma_{ij}^n} \in X
\text{ and  }e^{p^s}_{il}=\prod \limits _n d_{nl}^{\alpha_{in}}.$$
Thus, we have an extension $D$ of $M^*$ by $G$ in which $M^* \subseteq Z(D)$, $\overline{a_i}M^*$,$\overline{b_j}M^*$ correspond to $a_i$,$b_j$ respectively and 
\begin{align*}
    [\overline{a_i},\overline{a_j}]&=\prod \limits_{n=1}^k \overline{b_n}^{\gamma_{ij}^n}e_{ij} & (i,j=1,2,\ldots,d), \\
    \overline{a_i}^{p^s}&=\prod \limits_{n=1}^k \overline{b_n}^{\alpha_{in}} & (i=1,2,\ldots,d), \\
    [\overline{b_i},\overline{a_j}]&=d_{ij} & (1 \leq i \leq k, 1 \leq j \leq d),\\
[\overline{b_i},\overline{b_j}]&=\overline{b_i}^{p^{t_i}}=1& (1 \leq i,j \leq k).
\end{align*}
It remains to verify that $M^{*} \cap D' = M$. Observe that $M^{*}$ is
generated by the elements $e_{ij}$ and $d_{ij}$, whereas $D'$ is generated
by the elements $d_{ij}$ together with all the elements of the form
\[
\prod_{n=1}^{k} \overline{b_n}^{\,\gamma_{ij}^{n}} e_{ij}.
\]
Consequently, an element $f \in M^{*}$ lies in $D'$ if and only if it can be
expressed as
\[
f = \prod e_{ij}^{\lambda_{ij}} \prod d_{ij}^{\mu_{ij}},
\]
where the integers $\lambda_{ij}$ satisfy
\[
\sum_{i,j} \lambda_{ij}\gamma_{ij}^{n} \equiv 0 \pmod{p^{t_i}},
\qquad 1 \le n \le k.
\]
In particular, this shows that $N \subseteq D'$. Thus, it remains only to
verify that
\[
\sum_{i,j} \lambda_{ij}\gamma_{ij}^{n} \equiv 0 \pmod{p^{t_i}}
\quad (1 \le n \le k)
\]
if and only if
\[
\sum_{i,j} \lambda_{ij}(a_jG' \wedge a_iG') \in \ker \rho.
\]
This equivalence follows directly from the identity
\[
\rho(a_jG' \wedge a_iG') = [a_j,a_i],
\]
and hence the desired result holds.
\end{enumerate}
\qed
\begin{remark}\label{kernel=X}
    The proof shows that $X$ is the kernel of $$\lambda : V \otimes W \to \gamma_3(F)/[R,F] \cap \gamma_3(F),$$ where $\lambda$ is given in Lemma \ref{Karpilovsky lambda map}.
\end{remark}The following result follows from Theorem \ref{maintheorem Schur computation}.

\begin{corollary}\label{Schurmultiplier_final}
        Let $G$ be a $p$-group of nilpotency class $2$ such that $G/G'$ is an $s$-elementary abelian $p$-group and $X$ is defined as above. Then $$|M(G)|=|\frac{V \otimes W}{X}| |\frac{V\wedge V}{W}|.$$
\end{corollary}

\section{
Achieving abelian groups as the Schur multiplier of non-abelian groups}
\label{Kourovka}

As a generalization of the abelian tensor product, Brown and Loday \cite{Brown1987} first introduced the non-abelian tensor product $G \otimes H$ of two groups $G$ and $H$. 
\begin{definition} \cite[Definition 2.1]{Brown1987} \label{definition: group tensor product and exterior product}
    Let $G$ and $H$ be groups which act on itself by conjugation, $g'^g=g^{-1}g'g$, 
and each of which acts upon the other in such a way that the following compatibility conditions hold: 
\begin{align*}
    g'^{(h^g)} = ((g'^{g^{-1}})^h)^g,
    h'^{(g^h)}=((h'^{h^{-1}})^g)^h 
\end{align*}
for all $g, g’ \in G$ and $h, h’ \in H$. Then the non-abelian tensor product $G \otimes H$ is the group generated by the symbols $g \otimes h$ and defined by the relations 
\begin{align*}
    gg’ \otimes h= (g^{g'} \otimes h^{g'})(g' \otimes h), \\ 
    g \otimes hh'= (g \otimes h')(g^{h'} \otimes h^{h'}) 
\end{align*}
for all $g, g’ \in G$ and $h, h’ \in H$. 
\end{definition}
\begin{definition} \label{definition: group exterior square}
    As a special case, let $G$ act on itself by conjugation. Then the above definition gives the non-abelian tensor square $G \otimes G$. The non-abelian exterior square of $G$ is defined as $G \wedge G=(G \otimes G) / \langle g \otimes g\mid g \in G\rangle$. For $g,h \in G$, the element $g \wedge h$ denotes the image of $g \otimes h$ in $G \wedge G$.
\end{definition}
It follows that, the map $f: G \wedge G \to G'$ given on the generators by $f(g \wedge h)=[g,h]$, $g,h \in G$, defines a surjective homomorphism.  Miller proved that $M(G) \cong \mathrm{Ker}(f)$ in \cite[Theorem 3]{Mill1952}, .
\begin{remark} \cite[Proposition 2.4]{Brown1987} \label{group tensor product and non-abelian tensor product are same}
    When $G$ and $H$ act on each other trivially, then $G \otimes H = G^{ab} \otimes _{\mathbb{Z}}H^{ab}$ where the latter is the tensor product of $\mathbb{Z}$-modules. 
\end{remark}

Later Rocco \cite{Rocco1991} gave another construction for $G \otimes G$. We recall this construction here to find the structure of the Schur multiplier of $G$.
Let $G$ and $G^\phi$ be isomorphic groups via the map $ g \mapsto g^\phi ,g\in G.$ Consider the group
$$\nu(G):=\langle G,G^\phi \mid R,R^\phi,[g_1,g_2^\phi]^{g_3}=[g_1^{g_3},(g_2^{g_3})^\phi]=[g_1,g_2^\phi]^{g_3^\phi} ,g_1,g_2,g_3 \in G \rangle,$$
where $R$ and $R^{\phi}$ are defining relations of $G$ and $G^{\phi}$ respectively.
Note that $[G,G^\phi]=\langle[g,h^\phi] \mid g,h \in G\rangle$ is a normal subgroup of $\nu(G).$ The map $\gamma: G \otimes G \to [G,G^\phi]$, defined on the generators by $\gamma(g_1 \otimes g_2)=[g_1,g_2^\phi]$, is an isomorphism by \cite[Proposition 2.6]{Rocco1991}. Moreover, the map $\gamma$ induces an isomorphism between $G \wedge G:=(G \otimes G)/\langle g \otimes g \mid g \in G \rangle$ and $[G,G^\phi]/\langle[g,g^\phi] \mid g \in G\rangle.$ 

The following result is from \cite[Lemma 2.1, 2.2, 2.3]{Rocco1991} and \cite[Lemma 9]{Blyth2009}.
\begin{lemma} \label{Rocco remark} 
    For a group $G$, the following relations hold in $\nu(G)$.
    \begin{enumerate}
        \item [(i)] If $G$ is nilpotent of class $c$, then $\nu(G)$ is nilpotent of class at most $c+1$.
        \item[(ii)] If $G$ is a $p$-group, then $\nu(G)$ is a $p$-group.
        \item[(iii)] If either $g \in G'$ or $h \in G'$, then $[g,h^\phi]=[h,g^\phi]^{-1}$.
        \item[(iv)] $[g,g^\phi]=1$ for all $g \in G'.$ 
        \item[(v)] $[g,g^\phi]$ is central in $\nu(G)$ for all $g \in G.$ 
        \item[(vi)]  $[g_1,g_2^\phi][g_2,g_1^\phi]$ is central in $\nu(G)$ for all $g_1,g_2 \in G.$
        \item[(vii)] If $g_1,g_2,g_3\in G$ such that $[g_3,g_1]=1=[g_3,g_2]$, then $[g_1,g_2,g_3^\phi]=1$.
        \item[(viii)] If $g_1,g_2\in G$ such that $[g_1,g_2]=1$, then $[g_1^{n},g_2^\phi]=[g_1,g_2^\phi]^n=[g_1,(g_2^\phi)^n]$, for all $n \in \mathbb{Z}$.
        \item [(ix)] $[[g_1,g_2^\phi],[h_1,h_2^\phi]]=[[g_1,g_2],[h_1,h_2]^\phi]$ for all $g_1,g_2,h_1,h_2 \in G$.
        \item [(x)] For all $g_1,g_2 \in G,[g_1,g_2^\phi]=[g_2,g_1^\phi]^{-1}$ in $G \wedge G$.
    \end{enumerate}
\end{lemma}

\begin{lemma}
    If $G$ is of nilpotency class $2$, then $G \otimes G$ is abelian.
\end{lemma}
\begin{proof}
It follows that ${[[g_1,g_2^\phi],[h_1,h_2^\phi]]=1}$ for all $g_1,g_2,h_1,h_2 \in G$ by applying Lemma \ref{Rocco remark}$(vii)$ in Lemma \ref{Rocco remark}$(ix)$. Hence, $G \otimes G$ is abelian.
\end{proof}

\begin{proposition} \cite[Proposition 20]{Blyth2009}\label{exterior-gen}
    Let $G$ be a polycyclic group with a polycyclic generating sequence $g_1,g_2,\ldots,g_k.$ Then $G \wedge G$ is generated by 
    $$\{[g_i,g_j^\phi],i >j\}.$$
\end{proposition}

Recall that $(a,b,c)$ denotes the element $aG' \otimes [b,c]+bG' \otimes [c,a] + cG' \otimes [a,b]$ in $V \otimes W$ where $a,b,c \in G$.
 Moreover, $X_1$ and $X_2$ are the subgroups of $V \otimes W$ generated by the elements $(a,b,c)$ for all $a,b,c \in G$ and $aG' \otimes a^{p^s}$ for $a \in G$ respectively. The proof of the following lemma is immediate and is therefore omitted.
\begin{lemma}
    Let $G$ be a $s$-special $p$-group minimally generated by the set $\{\alpha_1,\alpha_2,\ldots,\alpha_d\}$. Then $X_1$ is generated by the set
  $\{(\alpha_i,\alpha_j,\alpha_k) \mid 1 \le i <j<k\le d\}$.
\end{lemma}

The proof of the following lemma is similar to the proof of \cite[Proposition 3.3]{Rai2018}.

\begin{lemma} \label{basis$X_2$}
    Let $G$ be a $s$-special $p$-group. Suppose the sets $\{v_1,v_2,\ldots,v_d\}$ and $\{f(v_1),f(v_2),\ldots,f(v_r)\}$ are bases of $V$ and $G^{p^s}$ respectively. Then the following set forms a basis of $X_2:$
     
$\{v_i \otimes f(v_i),v_j\otimes f(v_i),v_i \otimes f(v_k)+v_k \otimes f(v_i)\mid 1 \le i \le r, i < k \le r, (r+1) \le j \le d \}.$
\end{lemma}

\begin{lemma} 
     Let $G$ be a $p$-group of nilpotency class $2$, and $G/G'$ be an $s$-elementary abelian $p$-group. Then the following sequence is exact.
     \begin{equation} \label{exactseq X}
        1 \to X \xrightarrow{\alpha} (G/G') \otimes G' \xrightarrow{\lambda}\ M(G) \xrightarrow{\gamma} M(G/G') \xrightarrow{\delta} G' \to 1. 
     \end{equation} 
\end{lemma} 
\begin{proof}
  The proof follows from Lemma \ref{Karpilovsky lambda map} and Remark \ref{kernel=X}.
\end{proof}

The following proposition extends \cite[Proposition 10]{Rai2023} given by Rai.

\begin{proposition} \label{Schurspecial}
Let $d,k, s \in \mathbb N$. For  $d \ge 2$ and $0 \le k \le d-2$,  $G_k$ denotes a $p$-group of nilpotency class $2$  minimally generated by the set 
$\{\alpha_1,\alpha_2,\ldots,\alpha_d\}$ such that $G_k'$ and $G_k/G_k'$ are $s$-elementary abelian and $o(\alpha_i)=p^s$ for all ${1 \le i \le d}$.  Suppose the set 
   ${\{[\alpha_i,\alpha_{i+1}] \mid i \in \{1,\ldots,  k+1\} \}}$ is independent in $G_k'$ and other commutators are trivial. Then  $M(G_k)\cong (\Z_{p^s})^n$, where ${n=\binom{d}{2}+2k+1}$.
\end{proposition} 

\begin{proof}
 Since $G_k$ is a group of nilpotency class $2$, the exponent of $M(G_k)$ is at most $p^s$ by \cite[Corollary~2.6]{Jones1974}.
Since $G_k/G_k' \cong (\Z_{p^s})^d$ and ${G_k' \cong (\Z_{p^s})^{k+1}}$, both Ker$(\gamma)$ and Im$(\gamma)$ are $s$-elementary abelian by the exactness of the sequence \eqref{exactseq X}. Hence, $M(G_k)$ is an $s$-elementary abelian $p$-group.
Following the proof of \cite[Proposition 10]{Rai2023}, it is easy to check that  
\begin{align*}
    D_{G_k}=\{(\alpha_i,\alpha_{i+1},\alpha_j)| 1\leq i \leq k+1, 1\leq j \leq d,  j \notin \{ i-1,i,i+1\}\}
\end{align*}
is a basis of $X$. Note that each element of $X$ has order $p^s$. So $|X|=p^{s((d-3)(k+1)+1)}$. 
By the exactness of the sequence \eqref{exactseq X}, we have 
\[|M(G_k)|=\frac{|M(G_k/G_k')|}{|G_k'|}\frac{|(G_k/G_k') \otimes G_k'|}{|X|}.\]
Since $|M(G_k/G_k')|=p^{\frac{1}{2}sd(d-1)}$ and $|(G_k/G_k') \otimes G_k'|=p^{sd(k+1)}$, the result follows.
\end{proof}
The following proposition extends \cite[Proposition 11]{Rai2023} given by Rai.
   
\begin{proposition}
Let $d,k,s \in \mathbb N$. For  $d \ge 4$ and $1 \le k \le d-3$,  $G_k$ denotes a $p$-group of nilpotency class $2$  minimally generated by the set 
$\{\alpha_1,\alpha_2,\ldots,\alpha_d\}$ such that $G_k'$ and $G_k/G_k'$ are $s$-elementary abelian and $o(\alpha_i)=p^s$ for all $1 \le i \le d$.  If the set 
   ${\{[\alpha_i,\alpha_{i+1}] \mid i \in \{1,\ldots,  k\}\cup\{k+2\}\}}$ is independent in $G_k'$ and other commutators are trivial, then  $M(G_k)\cong (\Z_{p^s})^n$, where ${n=\binom{d}{2}+2k}$.
 \end{proposition} 
 \begin{proof}
Observe that the following set 
\begin{equation*}
\small
D_{G_k}=\begin{array}{cl}
\Bigl\{
(\alpha_i,\alpha_{i+1},\alpha_j)
\;\Big|\;
\begin{array}{l}
i\in \{1,2,\ldots,k\}\cup\{k+2\},\\
 1\leq j \leq d,  j \notin \{ i-1,i,i+1\}\end{array}
\Bigr\}
\displaystyle
 \ \bigcup \ 
\Bigl\{
(\alpha_{k+1},\alpha_{k+2},\alpha_{k+3})
\Bigr\}
\end{array}
\end{equation*}
is a basis of $X$. The proof proceeds similarly to the proof of Proposition~\ref{Schurspecial}.
\end{proof}
\begin{lemma}  \label{SchurG_{j,k}} 
Let $d,j,k \in \mathbb{N}$ such that $d \geq 2$, $0 \leq k \leq d-2$, ${1 \leq j \leq k+1}$.
   Consider a group $G_{j,k}^d$, minimally generated by $\alpha_1,\alpha_2,\ldots,\alpha_d$, having the following presentation,
\[G_{j,k}^d=\Biggl\langle 
\begin{array}{l|cl}
\alpha_i, 1 \le i \le d,	&  [\alpha_i,\alpha_{i+1}]=\beta_i, 1 \leq i \leq k+1,   \\
	\beta_u,1 \le u \le k+1 &  \alpha_i^{p^s}=\beta_i, \alpha_l^{p^s}=1 \text{ for } 1 \leq i \leq j <l \leq d, \\
     &  \beta_i^{p^{t_i}}= \beta_l^{p^s}=1 \text{ for }  1 \leq i \leq j < l \leq k+1
\end{array}
\Biggr\rangle.\]
\begin{enumerate} 
    \item [(i)]  If $j < k+1$, then 
$$M(G_{j,k}^d) \cong (\Z_{p^s})^{\binom{d}{2}+(2k+1)-3j} \times \prod_{i=1}^{j}\Z_{p^{(s-t_i)}} .
$$
\item [(ii)] If $j=k+1$, then 
$$M(G_{j,k}^d) \cong (\Z_{p^s})^{\binom{d}{2}-(k+1)} \times\prod_{i=1}^{k+1}\Z_{p^{(s-t_i)}} .$$
\end{enumerate}
\end{lemma}

\begin{proof}
For notational convenience, we write $G$ instead of $G_{j,k}^d$. We consider two cases.

\textbf{Case(i) $(j < k+1)$:}
Here $V \cong (\Z_{p^s})^d$ and $W \cong \prod_{i=1}^{j}\Z_{p^{t_i}}  \times (\Z_{p^s})^{k+1-j}.$ 
Following the proof of \cite[Proposition 10]{Rai2023}, we have that
$$\{(\alpha_i,\alpha_{i+1},\alpha_u)| 1\le i \le k+1,1\le  u\le d \text{ and } u \neq i-1,i,i+1\}$$
is a basis of $X_1$. Moreover, by Lemma \ref{basis$X_2$}, the following set is a basis of $X_2$,
\begin{equation}\label{eq:basis of X2}
    \begin{aligned}
    & \{\alpha_ l G' \otimes \alpha_l^{p^s},\alpha_mG' \otimes \alpha_l^{p^s},\alpha_ l G' \otimes \alpha_r^{p^s}+\alpha_ r G' \otimes \alpha_l^{p^s} \mid 1 \le l \le j,l <r \le j,\\
  & j+1 \le m \le d\}. 
\end{aligned}
\end{equation}
Now we claim that, $\{\alpha_ r G' \otimes \alpha_l^{p^s} \mid 1 \le r \le d, 1\le l \le j\} $ is contained in $ X$. Observe that, for $r \ge l+2$, $\alpha_l G' \otimes \alpha_r^{p^s}=\alpha_l G' \otimes [\alpha_r,\alpha_{r+1}] \in X_1.$ Hence $\alpha_ l G' \otimes \alpha_r^{p^s},\alpha_ r G' \otimes \alpha_l^{p^s} \in X$ for $r \ge l+2$. Now suppose $r=l+1$. consider the element,
$(\alpha_l ,\alpha_{l+1},\alpha_{l+2})=\alpha_ l G' \otimes \alpha_{l+1}^{p^s}+\alpha_{l+2} G' \otimes \alpha_l^{p^s}$.
For $j \ge l+2$, we have already seen that $\alpha_ {l+2} G' \otimes \alpha_{l}^{p^s} \in X$ and if $j < l+2$, then $\alpha_ {l+2} G' \otimes \alpha_{l}^{p^s} \in X_2$ by \eqref{eq:basis of X2}. Therefore, $\alpha_ l G' \otimes \alpha_{l+1}^{p^s} \in X$ and this proves our claim. Using this claim, it follows that 
\begin{equation*}
\small
\begin{array}{cl}
\Bigl\{
(\alpha_i,\alpha_{i+1},\alpha_u)
\;\Big|\;
\begin{array}{l}
i = j+1, j+2, \ldots, k+1,\\
1\le u \le d,
u \neq i-1,i,i+1
\end{array}
\Bigr\}
\displaystyle
 \ \bigcup \ 
\Bigl\{
\alpha_r G' \otimes \alpha_l^{p^s}
\;\Big|\;
\begin{array}{l}
1 \le r \le d,\\
1 \le l \le j
\end{array}
\Bigr\}
\end{array}
\end{equation*}
   is a basis of $X$. It is easy to check that in the first set each element $(\alpha_i,\alpha_{i+1},\alpha_u)$ has order $p^s$ and in the second set each element  $\alpha_ r G' \otimes \alpha_l^{p^s}$ has order $p^{t_l}$. Consequently, $log_p|X|=s\{(k+1-j)(d-3)+1\}+ d(\sum \limits_{i=1}^{j}t_i) $. Then applying Corollary \ref{Schurmultiplier_final}, we obtain
\begin{align*}
    |M(G)|
          & = p^{s\{\binom{d}{2}+(2k+1)-3j\}}\prod_{i=1}^{j}p^{(s-t_i)}.
\end{align*}
Now we compute the non-abelian exterior square $G \wedge G$ to compute the structure of the Schur multiplier. 
By Proposition \ref{exterior-gen}, $G \wedge G$ is generated by
\[\left\{
\begin{array}{cl}
    &[\alpha_e,\alpha_f^{\phi}] \text{ for } 1\le e <f \le d,\\
    &[\beta_u,\alpha_e^{\phi}]\text{ for } 1\le u\le k+1, 1\le e \le d, \\
    &[\beta_u,\beta_v^{\phi}]\text{ for }  1\le u<v\le k+1.
\end{array}
\right\}\]
Using Lemma \ref{Rocco remark},  the following relations hold in $G \wedge G$, 
\begin{align*}
   [\beta_u,\alpha_e^{\phi}]&=[\alpha_u^{p^s},\alpha_e^{\phi}]=[\alpha_u,\alpha_e^{\phi}]^{p^s}, &\text{ for } 1 \le u \le j, 1 \le e \le d;\\
  [\beta_u,\alpha_e^{\phi}]&=[\alpha_u,\alpha_{u+1},\alpha_e^{\phi}]=1, &\text{ for }  j+1 \le u \le k, 1 \le e \le d \text{ and } \\
  & &e \neq u-1,u,u+1,u+2;\\
   [\beta_{k+1},\alpha_e^{\phi}]&=[\alpha_{k+1},\alpha_{k+2},\alpha_e^{\phi}]=1, &\text{ for } 1 \le e \le d \text{ and } e \neq k,k+1,k+2;\\
   [\beta_u,\beta_v^{\phi}]&=[\alpha_u,\alpha_{u+1},\beta_v^{\phi}]=1, & \text{ for }  1\le u<v\le k+1.
\end{align*}
Also,
\begin{align*}
    [\alpha_{u+1}^{-1},\alpha_u^{-1},\alpha_{u+2}^\phi]^{\alpha_u}
    & = [\alpha_u\alpha_{u+1}\beta_u^{-1}\alpha_{u+1}^{-1}\alpha_u^{-1},\alpha_{u+2}^\phi]^{\alpha_u}= [\beta_u^{-1},\alpha_{u+2}^\phi]^{\alpha_u}\\
    & = [(\alpha_u^{-1}\beta_u^{-1}\alpha_u),(\alpha_u^{-1}\alpha_{u+2}\alpha_u)^\phi]\\
    & = [\beta_u^{-1},\alpha_{u+2}^\phi]= [\beta_u,\alpha_{u+2}^\phi]^{-1}.
    \end{align*} 
Similarly, we obtain $[\alpha_{u+2},\alpha_{u+1},\alpha_u^{\phi}]^{\alpha_{u+1}^{-1}} = [\beta_{u+1},\alpha_u^\phi]^{-1}$.
By the Hall-Witt identity, we have
\begin{align*}
    1 & =[\alpha_{u+2},\alpha_{u+1},\alpha_u^{\phi}]^{\alpha_{u+1}^{-1}}[\alpha_{u+1}^{-1},\alpha_u^{-1},\alpha_{u+2}^\phi]^{\alpha_u}[\alpha_u,\alpha_{u+2}^{-1},(\alpha_{u+1}^{-1})^\phi]^{\alpha_{u+2}}\\
    & = [\beta_{u+1},\alpha_u^\phi]^{-1} [\beta_u,\alpha_{u+2}^\phi]^{-1}.
\end{align*}
This implies that $[\beta_{u+1},\alpha_u^\phi]=[\beta_u,\alpha_{u+2}^\phi]^{-1}$ for $1 \le u \le k$.
Hence $G \wedge G$ is generated by 
\begin{equation}\label{eq:generators of G wegde G}
\left\{
\begin{array}{cl}
    & [\alpha_e,\alpha_f^{\phi}] \text{ for }  1\le e <f \le d,\\
    & [\beta_u,\alpha_e^{\phi}]\text{ for } j+1\le u\le k, e=u,u+1,u+2,\\
    & [\beta_{k+1},\alpha_e^{\phi}] \text{ for } e=k+1,k+2.
\end{array}
\right\}
\end{equation}
Since $|M(G)| =p^{s\{\binom{d}{2}+(2k+1)-3j\}} \prod\limits_{i=1}^{j} p^{(s-t_i)}$ and $|G'|=p^{s(k+1-j)+\sum\limits_{i=1}^{j}t_i}$, we have that $|G \wedge G|=p^{s\{\binom{d}{2}+3(k-j)+2\}}$.

 Consider the normal subgroup $H_u$ of $G$ generated by the following set
 \begin{align*}
     \{  \alpha_1 , \ldots , \alpha_{u-1} ,   \alpha_{u+3} ,\ldots  ,\alpha_d, \ & \beta_1, \ldots,\beta_{u-1}, \beta_{u+2},\ldots,\beta_{k+1} \}, \\ &j+1 \leq  u < k+1.
 \end{align*}
Then the group $B_u:=G/H_u$ has the following presentation
\[ B_u=\biggl\langle 
\begin{array}{l|cl}
	 \alpha_i,u \le i \le u+2, &  [\alpha_{u},\alpha_{u+1}]=\beta_u, [\alpha_{u+1},\alpha_{u+2}]=\beta_{u+1}   \\
	\beta_u,\beta_{u+1} &\alpha_i^{p^s}=\beta_u^{p^s}=\beta_{u+1}^{p^s}=1 
    \end{array}
\biggr\rangle.\]
By Proposition \ref{Schurspecial}, we have that $M(B_u)\cong (\Z_{p^s})^6$. Therefore, ${B_u \wedge B_u}$ has order $p^{8s}$ and is minimally generated by $\{[\alpha_{u},\alpha_{u+1}^\phi],[\alpha_{u+1},\alpha_{u+2}^\phi],[\alpha_{u},\alpha_{u+2}^\phi],$\\$[\beta_u,\alpha_{u}^\phi],[\beta_u,\alpha_{u+1}^\phi], 
[\beta_{u+1},\alpha_{u+1}^\phi],[\beta_{u+1},\alpha_{u+2}^\phi],[\beta_{u},\alpha_{u+2}^\phi]\}$ with each element of this set having order $p^s$. By the natural epimorphism
 $[G,G^\phi] \to [B_u,B_u^\phi]$, the set of elements $\{[\beta_u,\alpha_{u}^\phi],[\beta_u,\alpha_{u+1}^\phi],[\beta_{u},\alpha_{u+2}^\phi],[\beta_{u+1},\alpha_{u+1}^\phi],[\beta_{u+1},\alpha_{u+2}^\phi]\}$ have order at least $p^s$ in  $G \wedge G$. Since ${[\beta_i,\alpha_e^{\phi}]^{p^s}=[\beta_i^{p^s},\alpha_e^{\phi}]=1}$, for $i=u,u+1$ and ${u\le e\le u+2}$, the elements $[\beta_u,\alpha_{u}^\phi],[\beta_u,\alpha_{u+1}^\phi],[\beta_{u},\alpha_{u+2}^\phi],[\beta_{u+1},\alpha_{u+1}^\phi],$\\$[\beta_{u+1},\alpha_{u+2}^\phi]$ have order $p^s$.

Now consider the following quotient group, 
\begin{align*}
  A &= \frac{G}{\langle \beta_1, \ldots,\beta_{k+1}\rangle}  \cong \langle \alpha_1\rangle \times \langle \alpha_2 \rangle \times \ldots\times \langle \alpha_d \rangle \cong (\Z_{p^s})^d.
\end{align*}
Note that $A \wedge A$ is minimally generated by the set ${\{[\alpha_p,\alpha_{q}^{\phi}]\mid 1\le p <q \le d\}}$ and each element of this set  has order $p^s$. The natural epimorphism
 ${[G,G^\phi] \to [A,A^\phi]}$
 shows that each element of the set $\{[\alpha_p,\alpha_q^{\phi}]\mid 1 \le p <q \le d\}$  has order at least $p^s$ in $G \wedge G$. 
 
Since $|G \wedge G|=p^{s\{\binom{d}{2}+3(k-j)+2\}}$, using the above computations of the order of the generators, it follows  from \eqref{eq:generators of G wegde G}, that $G \wedge G$ is minimally generated by the set
\[\left\{ 
\begin{array}{cl}
	&  [\alpha_e,\alpha_f^{\phi}] \text{ for }  1\le e <f \le d,   \\
	&  [\beta_u,\alpha_e^{\phi}]\text{ for } j+1\le u\le k, e=u,u+1,u+2, \\
     &  [\beta_{k+1},\alpha_e^{\phi}] \text{ for } e=k+1,k+2 
\end{array}
\right\},\]
with each generator having order $p^s$. Therefore, $M(G)$ is minimally generated by the set
\[\left\{ 
    \begin{array}{cl}
    &  [\alpha_e,\alpha_{e+1}^{\phi}]^{p^{t_e}} \text{ for } 1 \le e \le j,\\
                 &  [\alpha_e,\alpha_f^{\phi}] \text{ for } 1 \le e < f \le k+2 \text{ and } f \neq e+1, \\
                 & [\alpha_e,\alpha_f^{\phi}] \text{ for } k+2 \le e < f \le d,\\
                 & [\beta_u,\alpha_e^{\phi}]\text{ for } j+1\le u\le k, e=u,u+1,u+2,\\
                 & [\beta_{k+1},\alpha_e^{\phi}] \text{ for } e=k+1,k+2
\end{array}
\right\}.\]
Hence, $M(G) \cong (\Z_{p^s})^{\binom{d}{2}+(2k+1)-3j} \times \prod\limits_{i=1}^{j}\Z_{p^{(s-t_i)}}$.

\textbf{Case(ii) $(j=k+1)$:} The proof goes similarly as the proof of case (i). Note that, here $V \cong (\Z_{p^s})^d$ and $W \cong  \prod\limits_{i=1}^{k+1}\Z_{p^{t_i}}$. In this case, $X_2$ has a basis 
\[ \biggl\{ 
\begin{array}{l|cl}
\alpha_ l G' \otimes \alpha_l^{p^s},\alpha_mG' \otimes \alpha_l^{p^s}, &  1 \le l <r \le (k+1),  \\
	\alpha_ l G' \otimes \alpha_r^{p^s}+\alpha_ r G' \otimes \alpha_l^{p^s} & k+2 \le m \le d  
    \end{array}
\biggr\}.\]
We also have
$$\{\alpha_ r G' \otimes \alpha_l^{p^s} \mid 1 \le r \le d, 1\le l \le k+1\} \subseteq X$$
such that each element of this set has order $p^{t_l}.$
Since $log_p|V \otimes W|=d(\sum\limits_{i=1}^{k+1}t_i)$, we obtain that  $\{ \alpha_ r G' \otimes \alpha_l^{p^s} \mid 1 \le r \le d, 1\le l \le k+1 \}$ is a basis of $X$. Therefore, 
\begin{align*}
    |M(G)|
          & = p^{s\{\binom{d}{2}-(k+1)\}} \prod_{i=1}^{k+1}p^{(s-t_i)}.
\end{align*}
Moreover,
    $$G \wedge G=\langle [\alpha_e,\alpha_f^{\phi}] \text{ for }  1\le e <f \le d \rangle,$$
and $[\alpha_e,\alpha_f^{\phi}]$ has order $p^s$ for $1\le e <f \le d$.
Thus $M(G)$ is minimally generated by the set
\[\left\{ 
    \begin{array}{cl} &  [\alpha_e,\alpha_{e+1}^{\phi}]^{p^{t_e}} \text{ for } 1 \le e \le k+1,\\
                 &  [\alpha_e,\alpha_f^{\phi}] \text{ for } 1 \le e < f \le k+2 \text{ and } f \neq e+1, \\
                 & [\alpha_e,\alpha_f^{\phi}] \text{ for } k+2 \le e < f \le d
    \end{array}
\right\}.\]
Hence, $M(G) \cong   (\Z_{p^s})^{\binom{d}{2}-(k+1)} \times \prod\limits_{i=1}^{k+1}\Z_{p^{(s-t_i)}}.$\\ 
\end{proof}

\begin{theorem} \label{bound M(G_{j,k})}
 Let $d, r, s , m_i \in \mathbb{N}$ with $1 \leq r < d$ and
$0 < m_i < s$ for $i \in \{1, \ldots,r\}$.
The following abelian groups occur as the Schur multiplier of some non-abelian group:
 $$ (\Z_{p^s})^{\binom{d}{2}-(d-t)} \times  \Z_{p^{m_1}} \times \Z_{p^{m_2}} \times \cdots \times \Z_{p^{m_r}},$$
 where $t$ achieves all the following possible values
\begin{enumerate}
\item [(i)] $t=1$, if $r=d-1$.
    \item [(ii)]  $ t=1,2,3$, if $r=d-2$.
 \item [(iii)]  $1  \le t \le 3(d-r-1)$ except ${t=3(d-r-1)-1}$, if $r< d-2$.
\end{enumerate}
\end{theorem}

\begin{proof}
  Consider the group $G_{j,k}^d$ given in Lemma \ref{SchurG_{j,k}} with $r \le j$ and
 \[t_i=\biggl\{ 
\begin{array}{lcl}
s-m_i, && \text{ if } 1 \le i \le r\\
s, && \text{ if }r<i\le j.
	     \end{array}
\]
 Then we have,
$$M(G_{j,k}^d) \cong (\Z_{p^s})^{n_{j,k}} \times\prod_{i=1}^{r}\Z_{p^{m_i}} $$
for some $n_{j,k} \in \mathbb{N}$.  
 For $r \le j \le (d-1)$ and $(r-1) \le k \le (d-2)$, define a matrix,
 $$A=[m_{j,k}] \in M_{d-r}(\mathbb Z),$$
having entries
\[ m_{j,k}=\biggl\{ 
\begin{array}{lcl}
n_{j,k}, && \text{if } j \leq k+1\\
0, && \text{otherwise}.
	     \end{array}
\]
For $-1 \leq a \leq d-r-2$,   define a set ${S_a=\{m_{j,k}\mid k-j=a\}}$. 
  Note that when $a=-1$, it gives the principal diagonal entries of $A$. By Lemma \ref{SchurG_{j,k}}$(ii)$, ${\{m_{(k+1),k} \mid (r-1) \le k \le (d-2)\}}$ is the set of all integers between ${\min S_{-1}=\binom{d}{2}-(d-1)}$ and $\max S_{-1}={\binom{d}{2}-r}$. 
  
If $r=d-1$, then $k=d-2$ and $m_{(d-1),(d-2)}=\binom{d}{2}-(d-1)$. Therefore,
$t=1$ and hence $\textit{(i)}$ follows.

For $r=d-2$, $k$ takes the values $d-3,d-2$. Using Lemma \ref{SchurG_{j,k}}, we have 
$$m_{(d-2),(d-3)}=\binom{d}{2}-(d-2),\quad m_{(d-1),(d-2)}=\binom{d}{2}-(d-1),$$
$$m_{(d-2),(d-2)}=\binom{d}{2}-(d-3).$$
Therefore, $t=1,2,3$ and hence $(ii)$ follows.

Now assume that $r<d-2$.  If $a=k-j$, then ${r \le j \le (d-2)-a}$ and it follows that ${a+r \le k \le d-2}$. 
Consider ${0 \le a \le (d-r-2)}$. By Lemma~\ref{SchurG_{j,k}}$(i)$, $m_{k-a,k}
=\binom{d}{2}-k+(3a+1)$. Since $\min {S_{a}}=\binom{d}{2}+3a-(d-3) $ and $ {\max{S_{a}}= \binom{d}{2}+3a -(a+r-1)}$, the set
$S_a={\{m_{k-a,k}\mid a+r \le k \le d-2\}}$  consists of all the integers $n$ such that
$$ \binom{d}{2}+3a-(d-3) \leq n \leq  \binom{d}{2}+3a -(a+r-1).$$
Also, for  $0\leq a \leq d-r-5$, we have $${\min{S_{a+1}}=\binom{d}{2}+3(a+1)-(d-3)\le \max{S_{a}}=\binom{d}{2}+3a -(a+r-1)}.$$ 
Thus, 
${\{m_{k-a,k} \mid 0\leq a \leq d-r-5,\ a+r \leq k \leq d-2\}}$  is the set of all integers between $\min{S_0}=\binom{d}{2}-(d-3)$ and ${\max{S_{d-r-5}}=\binom{d}{2}-(d-(3d-3r-9))}$. 
Hence, we are left with ${a=d-r-4}$, ${d-r-3}$ and $d-r-2$. The following table gives the values of $m_{j,k}$ for those $a$.
\[\setlength{\tabcolsep}{3pt} 
\renewcommand{\arraystretch}{2.2} 
{\footnotesize
\begin{tabular}{|c|c |} 
	\hline
	$a$ &  $m_{j,k}$ \\ [0.005ex] 
	\hline\hline 
    $d-r-4$ & $\binom{d}{2}-(d-(3d-3r-7)),~\binom{d}{2}-(d-(3d-3r-8)),~\binom{d}{2}-(d-(3d-3r-9))$\\
    \hline
    $d-r-3$ & $\binom{d}{2}-(d-(3d-3r-5)),~\binom{d}{2}-(d-(3d-3r-6)$\\
    \hline
    $d-r-2$ & $\binom{d}{2}-(d-(3d-3r-3))$\\
    \hline
    \end{tabular}}\]\\
Therefore, ${\{m_{k-a,k} \mid 0\leq a \leq d-r-2,\ a+r \leq k \leq d-2\}}$ is the set of all integers lying between $\binom{d}{2}-(d-3)$ and 
 $\binom{d}{2}-(d-(3d-3r-3))$ except $\binom{d}{2}-(d-(3d-3r-4))$.
 Now for $a=-1$, we have
 $$m_{(d-2),(d-3)}=\binom{d}{2}-(d-2),\quad m_{(d-1),(d-2)}=\binom{d}{2}-(d-1).$$ 
This proves $(iii)$ and hence the result.
\end{proof} 

The following corollary indicates that the number of copies of $\mathbb{Z}_{p^s}$ in the Schur multiplier obtained in Theorem \ref{bound M(G_{j,k})} can achieve all but finitely many values in $\mathbb{N}$.

\begin{corollary} \label{values of n}
Let $n,r,s,m_i \in \mathbb{N}$  for $i \in \{1,2,\cdots r\}$ such that    ${0 < m_i < s}$. Then the following abelian groups  occur as the Schur multiplier of some non-abelian groups:
     $$(\Z_{p^s})^n \times \Z_{p^{m_1}} \times \Z_{p^{m_2}} \times \cdots \times \Z_{p^{m_r}}$$ 
for
$n \ge \binom{a}{2}-(a-1)$, where $a=\lceil\frac{3r}{2}+2\rceil.$

\end{corollary}

\begin{proof}
If $b \in \mathbb N$ such that $(r+2) < b$, then taking $d=b$ in Theorem \ref{bound M(G_{j,k})}$(iii)$, $n$ obtains all the values between
 $\binom{b}{2}-(b-1)$ and $ \binom{b}{2}+2b-3(r+1)$ except $\binom{b}{2}+2b-3(r+1)-1$.
 Taking $d=b+1$, we get that $n$ obtains all the values between $\binom{b+1}{2}-b 
 $ and $\binom{b+1}{2}+2(b+1)-3(r+1)$   except $\binom{b+1}{2}+2(b+1)-3(r+1)-1$.
  But, for $b \ge \frac{3r}{2}+2$, we have 
    $$\binom{b+1}{2}-b \le \binom{b}{2}+2b-3(r+1)-1.$$
    By taking $d$ to be all values greater than $\lceil\frac{3r}{2}+2\rceil$, we obtain that 
    $n$ takes all the values greater equal to $\binom{a}{2}-(a-1)$ for $a=\lceil\frac{3r}{2}+2\rceil$.
\end{proof}


\noindent \textbf{Proof of Theorem \ref{Application_Kourovka}.}

   \noindent $(i)$ For $n, m_1 \in \mathbb N \cup \{0\}, m_1 <s$, there are $d, q, t_1 \in \mathbb N$ such that $n={d \choose 2}-q$ for $1 \leq q \leq d-1$ and $m_1=s-t_1$. 
   Taking the group $G_{q,q-1}^d$   given in Lemma~\ref{SchurG_{j,k}} with $t_i=s$ for  $2\leq i \leq j$, we have  ${M(G_{q,q-1}^d) \cong (\Z_{p^s})^{n} \times \Z_{m_1}}$.
   \\
   
   \noindent $(ii)$ By Corollary \ref{values of n}, $(\Z_{p^s})^n \times \Z_{p^{m_1}}\times \Z_{p^{m_2}}$ is the Schur multiplier of some non-abelian group for $n \ge 6$. Moreover, the cases $n=1$ and $n=3,4,5$ follow from Theorem \ref{bound M(G_{j,k})} $(ii)$ by taking $d=3$ and $d=4$ respectively.\\

   \noindent $(iii)$ 
This follows from Corollary \ref{values of n} for $r \geq 3$.\\

Let $n_1,n_2,n_3 \in \mathbb{N} \cup \{0\}$.  We show that the group $\Z_{p^{n_1}} \times \Z_{p^{n_2}} \times \Z_{p^{n_3}}$ is also obtained as the Schur multiplier of some non-abelian $p$-group. 
If all $n_1,n_2$ and $n_3$ are distinct or $n_2=n_3$ with $n_2 < n_1 $, then the result follows by taking $n=1,s=n_1,m_1=n_2,m_2=n_3$ in $(ii)$. If $n_2=n_3$ and $n_2 > n_1$, then  we obtain the result  by taking $n=2,s=n_2$ and $m_1=n_1$ in $(i)$. Finally, if $n_1=n_2=n_3$, then taking $n=3,s=n_1$ and $m_1=0$ in $(i)$ completes the proof.

\qed

We note that many more abelian $p$-groups can be obtained as the Schur multiplier of some non-abelian $p$-groups. Constructing the appropriate $G$ of nilpotency class $2$ having $G/G'$ as an $s$-elementary abelian $p$-group and using the  method described in Section~\ref{section: Schur multiplier of group with s elementary abelian} gives the result. The following table provides some classes of groups of such kind and their corresponding Schur multiplier, which can be derived using similar techniques used in the proof of Lemma \ref{SchurG_{j,k}}.


\[\setlength{\tabcolsep}{4pt} 
\renewcommand{\arraystretch}{3} 
{\small
\begin{tabular}{|c| c |} 
	\hline
	$G$ & $M(G)$ \\ [0.15ex] 
	\hline\hline 
    
    

    \makecell[l]{$\langle \alpha_1,\alpha_2,\alpha_3,\beta_1,\beta_2 \mid [\alpha_1,\alpha_2]=\beta_1,[\alpha_2,\alpha_3]=\beta_2,$ \\
$\alpha_1^{p^s}=\beta_1,\alpha_2^{p^s}=\beta_2,\alpha_3^{p^s}=\beta_1^{p^s}=\beta_2^{p^s}=1\rangle$\\
    $\times \langle \gamma \mid \gamma^{p^t}=1 \rangle$, $t \le s$}
    & $\Z_{p^s} \times (\Z_{p^t})^3$   \\ 
	\hline

    
    \makecell[l]{$\langle \alpha_1,\alpha_2,\beta_1 \mid [\alpha_1,\alpha_2]=\beta_1,\alpha_1^{p^s}=
   \alpha_2^{p^s}=\beta_1^{p^s}=1\rangle$\\
   $\times \langle \gamma \mid \gamma^{p^t}=1 \rangle$}
   & $ (\Z_{p^s})^2 \times (\Z_{p^t})^2$   \\ 
	\hline

    \makecell[l]{$\langle \alpha_1,\alpha_2,\alpha_3,\beta_1,\beta_2 \mid [\alpha_1,\alpha_2]=\beta_1,[\alpha_2,\alpha_3]=\beta_2,$ \\
$\alpha_3^{p^s}=\beta_1,\alpha_1^{p^s}=\alpha_2^{p^s}=\beta_1^{p^s}=\beta_2^{p^t}=1\rangle$}
& $(\Z_{p^t})^2 \times \Z_{p^{s-t}} \times \Z_{p^s}$\\
 \hline
    
  \makecell[l]{$\langle \alpha_1,\alpha_2,\alpha_3,\beta_1,\beta_2,\beta_3 \mid [\alpha_1,\alpha_3]=\beta_1,[\alpha_2,\alpha_3]=\beta_2,$ \\
$[\alpha_1,\alpha_2]=\beta_3,\alpha_1^{p^s}=\beta_1,\alpha_3^{p^s}=\beta_2,\alpha_2^{p^s}=\beta_1^{p^s}$ \\ 
$=\beta_2^{p^t}=\beta_3^{p^s}=1\rangle$}
& $\Z_{p^t} \times \Z_{p^{s-t}} \times (\Z_{p^s})^2$\\
 \hline

   \makecell[l]{$\langle \alpha_i,\beta_i,1 \le i \le 4 \mid [\alpha_1,\alpha_2]=\beta_1,[\alpha_2,\alpha_3]=\beta_2,$\\
$[\alpha_3,\alpha_4]=\beta_3,[\alpha_1,\alpha_4]=\beta_4,\alpha_1^{p^s}=\beta_1,\alpha_2^{p^s}=\beta_2,$ \\ 
$\alpha_3^{p^s}=\beta_3,\alpha_4^{p^s}=\beta_4,\beta_1^{p^{t_1}}=\beta_2^{p^{t_2}}=\beta_3^{p^s}=\beta_4^{p^s}=1\rangle$\\
for $t_1 \le t_2$}
& $\Z_{p^{s-t_1}} \times \Z_{p^{s-t_2}} \times (\Z_{p^{s+t_1}})^2$\\ 
\hline
 \makecell[l]{$\langle \alpha_1,\alpha_2,\alpha_3,\beta \mid [\alpha_1,\alpha_3]=\beta=[\alpha_1,\alpha_2],$  \\
    $\alpha_2^{p^s}=\beta^{p^t}=\alpha_3^{p^s},\alpha_1^{p^s}=\beta^{p^{2t}}=1\rangle$}
    & $\Z_{p^{s-2t}} \times (\Z_{p^t})^2 \times (\Z_{p^s})^2$   \\
    \hline

\end{tabular}}
\]

\section{The structure and the Schur multiplier of $s$-extraspecial $p$-groups} \label{s-extraspecial structure}
Recall that $s$-special $p$-groups of rank $k$ have $G/G'$ and $G'$ to be $s$-elementary abelian. Thus, Theorem \ref{maintheorem Schur computation} can be used to compute their Schur multiplier. In this section we obtain the Schur multiplier of $s$-extraspecial $p$-groups explicitly by proving a structural result for such groups.


The following result is a corollary of McCoy's theorem \cite{McCo1942} and will be useful in obtaining the structural result about $s$-extraspecial $p$-groups.
\begin{theorem}[Chapter I, Corollary I. 30 of \cite{McDo1984}] \label{McCoy theorem}
    Let $A$ be an $n \times n $ matrix over a commutative ring $R$. The system of linear equations $Ax=0$ has a non-trivial solution if and only if det$(A)$ is a zero divisor in R.
\end{theorem}

The following proof generalizes the arguments given in the proof of \cite[Theorem 3.3.4]{Karp1987}.

\noindent \textbf{Proof of Theorem \ref{Centralproductstructure}.}
    The case $|G| =p^{3s}$ is trivial. Now we assume that ${|G|  > p^{3s}}$. Since $G$ is $s$-extraspecial, we can find generators $\{g_1,\ldots,g_d\}$ such that $g_iG'$ has order $p^s$ for $i \in \{1,\ldots,d\}$ and ${G/G' \cong \langle g_1G' \rangle \times \ldots \times \langle g_dG'\rangle}$. 
    Without loss of generality, $g_i$ can be chosen such that $w:=[g_1,g_2]$ generates $G'=Z(G)$ and ${[g_1,g_i]=[g_2,g_i]=1}$ for ${i \in \{3,\ldots,d\}}$.
    We have that ${\langle w \rangle = Z(G)= G' \cong \mathbb{Z}_{p^s}}$. Let $G_1=\langle g_1,g_2 \rangle$. We define $K$ to be the centralizer of $G_1$ in $G$.
    We now show that $G_1$ is an $s$-extraspecial $p$-group of order $p^{3s}$. 
    We have  ${G_1'=G' = \langle w \rangle}$ and hence, 
    ${G_1/G_1'=\langle g_1G',g_2G'\rangle \cong \mathbb{Z}_{p^s} \times \mathbb{Z}_{p^s}}$.
    Therefore, $|G_1|=p^{3s}$, ${G_1' \cong \mathbb{Z}_{p^s}}$, and ${G_1/G_1' \cong \mathbb{Z}_{p^s} \times \mathbb{Z}_{p^s}}$. We are left to show that ${Z(G_1) \cong G_1'}$. Note that any element of $G_1$ can be written as $g_1^ag_2^bw^c$.
    If $g_1^ag_2^bw^c \in Z(G_1)$, then $[g_1^ag_2^bw^c,g_1]=1$ implies $[g_2,g_1]^b =1$. Thus ${b \equiv 0 }$ mod ${p^s}$. Similarly, $[g_1^ag_2^bw^c,g_2]=1$ implies $a \equiv 0$ mod $p^s$. Hence, $Z(G_1)  \subseteq \langle w \rangle = G_1'$. Therefore, $G_1$ is an $s$-extraspecial $p$-group. Note that since $Z(G)=G'$, we have that ${[G,G_1] \subseteq G' =Z(G)=Z(G_1)}$. 
    
    Next, we show that $G=G_1K$, i.e., for any $g \in G$, there exists $h \in G_1$ such that $gh^{-1}$ centralizes $G_1$. To prove this, we consider the inner automorphism $i_g:G \to G$. We aim to produce an $h \in G_1$ such that $i_g=i_h$ on $G_1$. 
    Since $[G,G_1] \subseteq Z(G_1)$,  $i_g$ acts trivially on $G_1/Z(G_1)$.
    To produce our required $h \in G_1$, we show that any automorphism $\psi$ of $G_1$ that acts trivially on $G_1/Z(G_1)$ is an inner automorphism of $G_1$. We have $G_1/Z(G_1) \cong \langle g_1 \rangle \times \langle g_2 \rangle$. Note that, for $i$ with $1 \leq i \leq 2$, ${\psi(g_i)=g_iw^{m_i}}$ for some $m_i$ with $0 \leq m_i < p^s$. Thus, there are at most $p^{2s}$ distinct automorphisms of $G_1$ which act trivially on $G_1/Z(G_1)$. We show that there are $p^{2s}$ distinct inner automorphisms of $G_1$ that act trivially on $G_1/Z(G_1)$. Towards this end, we consider a transversal $\{t_j \mid 1 \leq j \leq p^{2s}\}$ for $Z(G_1)$ in $G_1$. The corresponding inner automorphisms $\psi_{t_j}$ of $G_1$  leave each coset of $Z(G_1)$ invariant. 
Therefore, $\psi_{t_j}$ acts trivially on $G_1/Z(G_1)$. 
It is easy to note that $\psi_{t_j}$ are distinct. Therefore, we have that $G=G_1K$.

    
 We have  ${\langle w \rangle \subseteq G_1 \cap K \subseteq Z(K) =Z(G) = \langle w \rangle}$. Hence, ${G_1 \cap K =Z(G)}$. 
    Now, it remains to show that $K$ is an $s$-extraspecial $p$-group. If $K \subseteq G_1$, then $G=G_1K=G_1$ and $|G|=p^{3s}$,  a contradiction. Hence, $K \not\subseteq G_1$. Note that $K$ is non-abelian since $K$ being abelian implies that $K= Z(K) \subseteq G_1$.
    Observe that
    $\langle g_3G',\ldots,g_dG' \rangle \subseteq K/G'$, and hence $|K/G'|\geq p^{(d-2)s}$. Since 
    $G_1 \cap K =G'$, we have that $G/G'=G_1/G' \oplus K/G'$. Therefore,  $K/G'$ is an $s$-elementary abelian group.
    To show that $K$ is an $s$-extraspecial $p$-group, it is enough to show that $K'=Z(K)$.
    We have ${K' \subseteq G' =Z(K)=\langle w \rangle}$. If $K' \subsetneq Z(K)$, we show that there exists an element $g_3^{b_3}\ldots g_d^{b_d} \in Z(G)\setminus G'$, where $b_i \in \{0,\ldots,p^s-1\}$ such that ${b=[b_3, \ldots, b_d]^T \neq [0]}$. This would be a contradiction because $G$ is $s$-extraspecial. For all $i,j$ with $3 \leq i,j \leq d$, and $i\neq j$, $[g_i,g_j]=[g_1,g_2]^{a_{ij}}$ where ${a_{ij} \in \{0,\ldots,p^s-1\}}$. If for some $i,j$, $(a_{ij},p) = 1$, then $K'=G'$ and therefore, we assume that $(a_{ij},p) > 1$ for all $i,j$. 
    Note that for all $k$ with $3 \leq k \leq d$, $[g_3^{b_3}\ldots g_d^{b_d},g_k]=1$ if and only if $\sum \limits _{i=3}^{d}b_ia_{ik} \equiv 0$ mod $p^s$. This is equivalent to saying that $Ab=0$ mod $p^s$. Since $(a_{ij},p)>1$ for all $3 \leq i,j \leq d$, $p \mid \text{det}(A)$ and is therefore a zero divisor in $\mathbb{Z}_{p^s}$. Hence, Theorem~\ref{McCoy theorem} gives a non-trivial solution $b$ such that $Ab=0$ mod $p^s$. Thus, we have $g_3^{b_3}\ldots g_d^{b_d} \in Z(G) \setminus G'$. 
    Therefore, we have $K'=Z(K)$. By induction, $K$ is the central product of $s$-extraspecial subgroups $G_i$ of order $p^{3s}$ with $2 \leq i \leq r$ and that $|K|=p^{2(r-1)s+s}=p^{2rs-s}$. Thus, $G$ is the central product of $G_1,\ldots,G_r$ with $G_i \cap G_j=Z(G)$ and $|G|=p^{2rs+s}$.

    Now, by \cite[Theorem A]{HatVerYad2018}, it follows that $M(G)$ is an $s$-elementary abelian $p$-group of order $p^{(2r^2-r-1)s}$.
    \qed

Using the structural result \cite[Theorem 3.3.4]{Karp1987} for extraspecial $p$-groups, it can be shown that extraspecial $p$-groups of order $p^{2r+1}$, with ${r \geq 2}$ are not capable groups (cf~\cite[Theorem 3.3.6]{Karp1987}). As a corollary to Theorem \ref{Centralproductstructure}, a similar result can be obtained for $s$-extraspecial $p$-groups. The proof follows \textit{mutatis mutandis} the proof given in \cite{Karp1987}. We give an alternative proof here.
\begin{corollary} \label{unicentral}
    Let $G$ be an $s$-extraspecial $p$-group of order $p^{(2r+1)s}$, with ${r \geq 2}$. Then $G$ is unicentral. Consequently, $G$ is not a capable group.
\end{corollary}

\begin{proof}
By Theorem \ref{Centralproductstructure}, $|M(G)|=p^{(2r^2-r-1)s}$. Since $|W|=p^s$ and $|V|=p^{2rs}$, we have that $|V \otimes W|=p^{2rs}$ and $|V \wedge V|=p^{r(2r-1)s}$. Using Corollary \ref{Schurmultiplier_final}, we obtain that $|X|=p^{2rs}$ and $ X=G/G' \otimes G'$.  Hence $Z^*(G)=Z(G)$, by Remark \ref{kernel=X} and Lemma \ref{capability 2}.

\end{proof}

\section*{Acknowledgements} 
The authors acknowledge the support of the National Institute of Science Education and Research (NISER), Bhubaneswar and Homi Bhaba National Institute (HBNI), Mumbai.

\bibliographystyle{amsplain}
\bibliography{DP_refschursize}

\end{document}